\documentclass[aos,preprint]{imsart}

\RequirePackage[OT1]{fontenc}
\RequirePackage{amsthm,amsmath}
\RequirePackage[round]{natbib}
\RequirePackage[colorlinks,citecolor=blue,urlcolor=blue]{hyperref}
\RequirePackage{graphicx,soul}
\usepackage{mathtools}

\usepackage{amsfonts}
\usepackage{multicol}

\usepackage[margin=1in]{geometry}

\arxiv{}

\startlocaldefs
\numberwithin{equation}{section}
\theoremstyle{plain}
\newtheorem{thm}{Theorem}[section]
\newtheorem{assumption}{Assumption}[section]
\newtheorem{lemma}{Lemma}[section]
\theoremstyle{definition}

\endlocaldefs

\newcommand\bbeta{{\boldsymbol\beta}}
\newcommand\bSig{{\boldsymbol \Sigma}}
\newcommand\vep{{\varepsilon}}
\newcommand{\Tau}{\mathrm{T}}

\newcommand{\bA}{{\bf A}}

\newcommand{\bb}{{\bf b}}

\newcommand{\bff}{{\bf f}}

\newcommand{\bh}{{\bf h}}
\newcommand{\bI}{{\bf I}}

\newcommand{\mcM}{\mathcal M}
\newcommand{\bT}{{\bf T}}
\newcommand{\bu}{{\bf u}}

\newcommand{\bX}{{\bf X}}

\newcommand{\bY}{{\bf Y}}

\newcommand\mbK{{\mathbb K}}
\newcommand\mbR{{\mathbb R}}

\newcommand\mcF{{\mathcal F}}
\newcommand\mcH{{\mathcal H}}

\newcommand\mcU{{\mathcal U}}

\DeclareMathOperator\E{E}
\DeclareMathOperator\Cov{Cov}

\DeclareMathOperator\Var{Var}
\DeclareMathOperator\vc{vec}

\begin{document}

\begin{frontmatter}
\title{Optimal Function-on-Scalar Regression \\ over Complex Domains}
\runtitle{Functional Regression}

\begin{aug}
\author{\fnms{Matthew} \snm{Reimherr}\thanksref{s1}
\ead[label=e1]{mreimherr@psu.edu}
\ead[label=u1,url]{www.personal.psu.edu/mlr36}}
\author{\fnms{Bharath} \snm{Sriperumbudur}\thanksref{s2}
\ead[label=e2]{bks18@psu.edu}
\ead[label=u2,url]{http://personal.psu.edu/bks18/}}
\and \\
\author{\fnms{Hyun Bin} \snm{Kang}
\ead[label=e3]{h.kang@wmich.edu}
\ead[label=u3,url]{https://wmich.edu/statistics/directory/kang}}

\thankstext{s1}{Supported by NSF DMS 1712826}
\thankstext{s2}{Supported by NSF DMS 1713011}
\runauthor{Reimherr, Sriperumbudur, and Kang}

\affiliation{Pennsylvania State University and Western Michigan University}

\address{Matthew Reimherr \\
Department of Statistics \\
Pennsylvania State University \\
\printead{e1}\\
\printead{u1}}

\address{Bharath Sriperumbudur \\
Department of Statistics \\
Pennsylvania State University \\
\printead{e2}\\
\printead{u2}}

\address{Hyun Bin Kang \\
Department of Statistics \\
Western Michigan University \\
\printead{e3}\\
\printead{u3}}

\end{aug}

\begin{abstract}
In this work we consider the problem of estimating function-on-scalar regression models when the functions are observed over multi-dimensional or manifold domains and with potentially multivariate output.  We establish the minimax rates of convergence and present an estimator based on reproducing kernel Hilbert spaces that achieves the minimax rate.  To better interpret the derived rates, we extend well-known links between RKHS and Sobolev spaces to the case where the domain is a compact Riemannian manifold.  This is accomplished using an interesting connection to Weyl's Law from partial differential equations.
We conclude with a numerical study and an application to 3D facial imaging.
\end{abstract}


\begin{keyword}
\kwd{Functional Data Analysis, Reproducing Kernel Hilbert Space, Functional Regression, Optimal Regression, Weyl's Law}
\end{keyword}

\end{frontmatter}

\doublespacing

\section{Introduction}
Functional data analysis has seen a precipitous development in recent decades, in terms of methodology, theory, and applications.  As with classical statistics, functional linear regression models are used extensively in practice.  In recent years, there has also been a surge in the development of so-called \textit{next generation functional data analysis}, which involves functional data with highly complex structures.  Much of this development has been spurred by advances in biomedical imaging, where dense measurements are taken over various tissues, including the brain, arteries, eyes, and faces    \cite[e.g.,][]{ettinger2016spatial,lila2016smooth,kang2017manifold,choe2017comparing,lee2018bayesian}.  In each of these examples, the measurements are taken over complex spatial domains such as $\mbR^3$ or two-dimensional manifolds.   


Establishing the optimality of parameter estimates in FDA remains an important topic given the complexity of the data and models involved.  Indeed, depending on the problem, one can see a wide variety of convergence rates.  For example, in univariate mean estimation it was shown that the rates depend on the smoothness of the underlying parameter as well as the sampling frequency of the data; depending on how often the functions are sampled, one can obtain a parametric convergence rate or nonparametric convergence rate \citep{cai2011optimal,li2010uniform,zhang2016sparse}.  In scalar-on-function regression, the rates relate both to the smoothness of the slope function and the regularity of the predictor function; these rates have been extended to nonlinear models as well \citep{hall2007methodology,cai2012minimax,wang2015optimal,reimherr2017optimal,sun2017optimal}.  In high-dimensional function-on-scalar regression models it was shown that the convergence rates match the classic scalar outcome setting as long as the sampling is dense enough \citep{barber2017function,fan2017high}. In principal component estimation, one obtains rates that reflect how deep into the spectrum one wishes to estimate as well as how spread out the eigenvalues are \citep{dauxois1982asymptotic,jirak2016optimal,petrovich2017asymptotic}.  In each of these cases, different rates can be obtained depending on the regularity of the problem.     
However, optimality of function-on-scalar regression, especially with more complex domains and sampling schemes, has not yet been established.  Such results are critical given the recent developments of functional data methods involving manifolds \citep{kang2017manifold,dai2018principal,lin2018intrinsic}.  In this work we address this issue by: (1) establishing minimax lower bounds on the estimation rate (2) providing a minimax optimal estimator whose upper bounds match the developed lower bounds and (3) interpreting the rate via a new connection to Sobolev spaces over manifold domains.  

We develop our theory under a fairly general structure:
\begin{align}
Y_{ij\ell} 
= Y_{i\ell}(u_{ij}) + \delta_{ij\ell}
=  \sum_{p=1}^P X_{ip} \beta_{\ell p}(u_{ij}) + \vep_{i\ell}(u_{ij}) + \delta_{ij\ell}, \label{e:model}
\end{align}
for $i=1,\dots,n$, $j=1,\dots, m_i,$ and $\ell=1,\dots,L.$  Here $i$ indexes the subject, $j$ the observed domain point, and $\ell$ the coordinates of the functional outcomes.  
Intuitively, this means that for each subject we have $L$ functional outcomes, $Y_{i\ell}(u) \in \mbR$, that are only observed at domain points $u_{ij} \in \mcU$.  The domain $\mcU$ is most commonly the interval $[0,1]$, but it may also be a more complex manifold, both of which are included in our theory.  For example, in \cite{ettinger2016spatial} $\mcU$ represents the surface of an internal carotid artery, meaning that it is a two dimensional manifold sitting in a three dimensional space and $L=1$.  In \cite{kang2017manifold} they consider the shape of human faces, their framework results in $\mcU$ being a two dimensional manifold while $L=3$ since the face is measured in three dimensions.

The intrinsic dimension of $\mcU$ plays a critical role in the minimax estimation rates for $\beta_{\ell k}(u)$, while, interestingly, the value $L$ does not. In addition, it was previously thought that, in simpler settings, such as mean estimation, it was necessary to control the smoothness of the underlying functions $Y_{i \ell}(u)$ \citep{cai2012minimax}, or equivalently the errors $\vep_{ij}(u)$, however, we show that this is actually unnecessary and establish all of our results under the mild assumption that $\sup_{u \in \mcU}\Var(\vep_{i\ell}(u)) < \infty$, that is, the point-wise variance of the errors is bounded.      

We assume that $\beta_{\ell p}$ (for all $\ell,p)$ lie in a reproducing kernel Hilbert space (RKHS), and establish our rates relative to the rate of decay of the eigenvalues of the kernel defining the RKHS.  In contrast, \cite{cai2012minimax} develop theory for one dimensional mean estimation assuming the parameters lie in a particular Sobolev space, which will be included in our theory as a special example.  Under mild assumptions, we will show that the optimal rate of convergence is given by {\bf (HB: I don't think we've defined $m$ before this -- just had $m_i$)}
\[
O_P\left( (n m)^{-\frac{2h}{2h+1}} + n^{-1} \right),
\]
where $h$ is related to the kernel of the RKHS.  In Section \ref{s:hrate}, we consider the case where $\mcU$ is a compact $d$-dimensional Riemannian manifold. When the parameters $\beta_{\ell k}$ posses $r$ derivatives, we use a connection with Weyl's law to show that $h = r/d$, which extends well known results for Sobolev spaces on $\mathbb{R}^d$ \citep{Edmunds-96}
resulting in the rate
\[
O_P\left( (n m)^{-\frac{2r}{2r+d}} + n^{-1} \right),
\]
which clearly shows the effect of the intrinsic dimension of $\mcU$ on the convergence rates of our estimators, with higher dimensions leading to slower rates. This further highlights the utility in exploiting manifold structures that may reside in higher dimensional spaces; the convergence rate is tied only to the intrinsic dimension of the manifold and not to that of the ambient space.  

The remainder of the paper is organized as follows.  In section \ref{s:background} we provide an overview of the modelling assumptions and necessary mathematical tools.  In Section \ref{s:est} we define our estimation procedure and provide a formulation useful for establishing mathematical properties.  In Section \ref{s:theo} we collect our theoretical contributions, which constitute the primary novel contributions of the paper.  There we provide a general lower bound on the minimax rate, followed by a theorem showing that our proposed estimator achieves the optimal rate.  We conclude the section with discussion on the derived rate.  We provide a new connection between the eigenvalues of an RKHS and Sobolev spaces over manifold domains, which allow us to interpret our results in terms of the dimension of the domain and the smoothness of the parameters being estimated.  We conclude the paper with numerical work in Section \ref{s:num}, where we provide simulations that further articulate the rates seen in Section \ref{s:theo}.  We also provide an application to 3D facial imaging from anthropology, highlighting the utility of such tools in biomedical imaging.

\section{Background and Modeling Assumptions} \label{s:background}
Here we provide the necessary background as well as a clear outline of our modeling assumptions.
\subsection{Reproducing Kernel Hilbert Spaces}
RKHSs provide a variety of benefits for functional data analysis.  The first is that the kernel can be tailored to reflect certain beliefs or assumptions about the parameters, e.g. smoothness or periodicity.  The second is that the eigenfunctions of the kernel can be used as a basis for approximating functional observations and/or parameter estimates, though the reproducing property can also be used to obtain parameter estimates.  Lastly, commonly used spaces, such as Sobolev spaces, as well as estimation techniques such smoothing splines can naturally be viewed in an RKHS framework \citep{wahba1990spline,berlinet2011reproducing}.   

We assume throughout that $\mcU$ is a compact 
$d$-dimensional manifold with $d < \infty$, i.e., $\mcU$ is a second countable compact Hausdorff space such that each point $u \in \mcU$ is contained in an open set that is homeomorphic to an open set in $\mbR^d$. 
We assume that $\mcU$ is equipped with a countably additive measure, $\mu$, with respect to the Borel $\sigma$-algebra, whose support is $\mcU$ and satisfies $\mu(\mcU)=1$.  This means that we can define integrals over $\mcU$ and the space, $L^2(\mcU, \mu)$, of square integrable functions over $\mcU$ is equipped with the inner product
\[
\langle f, g \rangle = \int_{\mcU} f(u) g(u) \ d \mu(u).   
\]
Throughout, for notational simplicity, we will often write $L^2$ for $L^2(\mcU,\mu)$.  A kernel function, $K:\mcU \times \mcU \to \mbR^+$, is a bivariate function that is symmetric, positive definite, and continuous (though this can be relaxed).  There is a one-to-one correspondence between RKHSs and kernel functions.  One can generate the RKHS from $K$ in at least one of two ways, though for our purposes one in particular is especially useful \citep[Section 3.2]{berlinet2011reproducing}.  Note that any norm $\| \cdot\|$ or inner product $\langle \cdot, \cdot \rangle$ written without subscript is understood to be with respect to $L^2$.  By Mercer's theorem we can write
\[
K(u, s) = \sum_{k=1}^\infty \tau_k v_k(u)v_k(s),
\]
where $v_k \in L^2$ are orthonormal and $\{\tau_k\}$ is a positive, non-increasing, summable sequence, with the convergence holding in an absolute and uniform sense.  One can then obtain $\mbK$ as a subset of $L^2$:
\[
\mbK = \left\{
f \in L^2: \sum_{k=1}^\infty \frac{\langle f, v_k \rangle ^2 }{\tau_k}
 < \infty \right\}.
\]
Then $\mbK$ is an RKHS when equipped with the inner product $\langle f, g \rangle_{\mbK} = \sum_k \tau_k^{-1} \langle f, v_k \rangle \langle g, v_k \rangle$.  On a technical note, since $L^2$ is a set of equivalence classes one is implicitly taking $f \in \mbK$ to be the unique member of each class that is continuous.  This view is especially useful as it emphasizes how quickly the coordinates of $f$ must decay when expressed in the $\{v_j\}$ basis. This decay turns out to be critical for understanding and developing minimax rates.

\subsection{Modeling Assumptions}
We now state our modeling assumptions. We provide a summary at the end of this section for ease of reference. We begin with the linear relationship
\[
Y_{i j \ell} = \sum_{p=1}^P X_{ip} \beta_{\ell p}(u_{ij}) + \vep_{i \ell}(u_{ij}) + \delta_{ij\ell}.
\]
This represents the model for the underlying trajectories/surfaces, which are not completely observed, as we will discuss shortly. The parameters, $\beta_{\ell k}$ are assumed to lie within $\mbK$.  Regularity assumptions about $\beta_{\ell k}$ are introduced by making assumptions about $\mbK$, especially the rate at which the eigenvalues of $K$ converge to zero.  

Unlike in \cite{cai2012minimax}, we make only minimal assumptions about the regularity of $\vep_{i \ell}(u)$.  In particular, we establish our minimax rates under the mild assumption that the point-wise variance of the errors is bounded, $\sup_{u \in \mcU} \Var(\vep_{i\ell}(u)) < \infty$, which implies (and is only slightly stronger than) $\E\| \vep_{i \ell}\|^2 < \infty$.  In \cite{cai2012minimax} a much stronger assumption that $\E \| \vep_{i \ell}\|_{\mbK}^2 < \infty$ was made, which, by the reproducing property implies our assumption.  While seemingly innocent, this is an incredibly strong assumption that would actually preclude achieving optimal convergence rates in most settings.  Practically, the data is usually much rougher than the underlying mean parameters. However, requiring that they reside in the same space implies that the $\beta_{\ell p}$ can only be smoothed up to the smoothness of the data.  For example, if $\mcU = [0,1]$ and $\beta_{\ell p}$ possessed two derivatives, while $\vep_{i \ell}$ only possessed one, then the rate given by \cite{cai2012minimax} would be $(nm)^{-2/3} + n^{-1}$, however, as we will show, this rate can be improved to $(nm)^{-4/5} + n^{-1}$. 
Furthermore, in settings such as finance or geosciences, $\vep_{i \ell}$ might not possess any derivatives or be part of any RKHS (e.g. Brownian motion or Ornstein-Uhlenbeck process).   

We will treat $X_{ij}$ as deterministic.  The observed points $u_{ij}$ will be assumed to be iid draws from $\mcU$, with density (w.r.t.~$\mu$) that is bounded away from 0 and $\infty$.  We also assume that functional outcome is observed with error, namely $Y_{ij\ell} =  Y_{i \ell}(u_{ij}) + \delta_{ij \ell}$.  The error $\delta_{ij \ell}$ are assumed to be iid across $i$ and $j$, though they can be dependent in $\ell$.  We assume these errors are centered and have finite variance.  We now summarize all of the assumptions introduced in this section.

\begin{assumption}\label{a:main}
	We make the following modeling assumptions.
	\begin{enumerate}
		\item The observed data are $\{Y_{ij\ell}, u_{ij}, X_{i1},\dots, X_{iP}\}$ for $i=1,\dots,n$, $j=1,\dots,m_i$, and $\ell=1,\dots,L$.  
		\item The observed times/locations, $u_{ij}$, are iid elements of $\mcU$, a compact $d$-dimensionl manifold.  The space $\mcU$ is equipped with a countably additive measure $\mu$ (over the Borel $\sigma$-field) with $\mu (\mcU)=1$.  The random elements $u_{ij}$ are assumed to have a density (w.r.t.~$\mu$) which is bounded above and below (from 0).  
		\item The observed data satisfy the linear model
		\[
		Y_{i j \ell} = \sum_{p=1}^P X_{ip} \beta_{\ell p}(u_{ij}) + \vep_{i \ell}(u_{ij}) + \delta_{ij \ell}.
		\]
		\item The mean parameters reside within the RKHS, i.e.,  $\beta_{\ell p} \in \mbK$, with continuous kernel $K(u,s)$.  The eigenvalues of $K$ satisfy $\tau_k \asymp k^{-2h}$ for $h \geq 1$.  
		\item The sequences  $\vep_{i \ell} \in L^2$, $u_{ij} \in \mcU$, and $\delta_{ij \ell}$ are random and independent of each other.
		\item The covariates $X_{ip}$ are deterministic.  Define $\Sigma_X = n^{-1} \sum X_i X_i^\top$ and assume that smallest and largest eigenvalues are bounded away from $0$ for all large $n$:  $0 < \nu^{-1} \leq \sigma_{min}(\Sigma_X) \leq \sigma_{max}(\Sigma_X)  \leq \nu < \infty $.
		\item Assume that the predictors are bounded $|X_{ip}| \leq \zeta  < \infty$. 
		\item The $\delta_{ij \ell}$ represent the measurement error and are iid across $i$ and $j$, though potentially dependent across $\ell$.  They have mean zero and finite variance, $\Var \delta_{ij \ell} \leq M_\delta < \infty$, for some fixed $M_\delta \in \mbR$.
		\item The stochastic processes $\vep_{i \ell}$ are iid across $i$, though potentially dependent across $\ell$.  They are assumed to have mean zero and to satisfy $\sup_{u \in \mcU} \Var(\vep_{i \ell}(u)) \leq M_{\epsilon} < \infty$, for some fixed $M_{\epsilon} \in \mbR.$  
	\end{enumerate}
\end{assumption}

\section{Estimation Methodology} \label{s:est}
We assemble an estimate of each $\ell$ coordinate separately.
Define the $l$th target function as
\begin{align*}
\ell_{mn}^l (\bb) & = \frac{1}{n} \sum_{i=1}^n \frac{1}{m_i} \sum_{j=1}^{m_i} (Y_{ij l} - \bX_i^\top \bb (u_{ij}))^2 + \lambda \sum_{k=1}^p \| b_{k}\|^2_{\mbK} \\
& =\frac{1}{n} \sum_{i=1}^n \frac{1}{m_i} \sum_{j=1}^{m_i} (Y_{ijl} - \langle \bX_i^\top \bb, K_{u_{ij}} \rangle_\mbK)^2 + \lambda \| \bb\|^2_{\mbK},
\end{align*}
where $b_k\in \mbK$ and $\bb=(b_1,\dots,b_P)$ are the generic arguments of the target function and $\bX_i = (X_{i1},\dots, X_{iP})$ are the covariates for the $ith$ unit  .  
The minimizer $\hat{\bbeta}_l$, can be obtained in a closed form using operator notation (as opposed to the representer theorem).  We can take the derivative with respect to $\bb$ (in the $\mbK$ topology) as
\[
D \ell_{mn}^l(\bb) = \frac{1}{n} \sum_{i=1}^n \frac{1}{m_i}\sum_{j=1}^{m_i} -2(Y_{ijl} - \langle K_{u_{ij}}, \bX_i^\top \bb \rangle_{\mbK})K_{u_{ij}} \bX_i+ 2 \lambda \bb,
\]l
where $K_{u_{ij}}(u):=K(u_{ij},u)$.  Define $\bh_{nml} \in \mbK^{p}$ as 
\begin{align}
\bh_{nml} = \frac{1}{n} \sum_{i=1}^n \frac{1}{m_i}\sum_{j=1}^{m_i} Y_{ijl} K_{u_{ij}} \bX_i , \label{e:hnm}
\end{align}
and the linear operator $\bT_{nm}: \mbK^p \to \mbK^p$ as
\begin{align}
\bT_{nm}(\bff) = \frac{1}{n} \sum_{i=1}^n \frac{1}{m_i} \sum_{j=1}^{m_i} \bX_i \bX_i^\top \bff(u_{ij}) K_{u_{ij}}
= \frac{1}{n} \sum_{i=1}^n \frac{1}{m_i} \sum_{j=1}^{m_i} \langle \bX_i^\top \bff, K_{u_{ij}} \rangle_{\mbK} K_{u_{ij}} \bX_i. \label{e:Tnm}
\end{align}
Setting the derivative equal to zero we get the operator form for the estimator
\[
D \ell^l_{mn}(\bb) = - 2 \bh_{nm} + 2 \bT_{mn}\bb  + 2 \lambda  \bb 
= 0
\Longrightarrow \hat \bbeta_l = (\bT_{nm} + \lambda \bI)^{-1} \bh_{nml}.
\]
This operator form for $\hat\bbeta_l$ is convenient for asymptotic theory.  In Section \ref{s:comp} we will discuss an efficient computational approximation to $\hat \bbeta_l$.  Using the representer theorem for RKHS, it is also possible to obtain an alternative equivalent formulation for $\hat \bbeta_l$ that can be computed exactly, but requires solving large systems of linear equations that can make it impractical for larger datasets.

\section{Theoretical Results} \label{s:theo}
We now provide our key theoretical results.  The first is a lower bound on the best possible estimation rate.  This bound is obtained using an application of Fano's lemma.  Second, we provide an estimator whose upper bound matches the lower bound, implying that it is optimal in a minimax sense.  Lastly, we provide an interpretation of the resulting rate by making a connection to Sobolev spaces with domains consisting of compact Riemannian manifolds.  

\subsection{Lower Bound}
Recall that when referring to a minimax rate, we have to specify the loss function as well as the class of models we are considering.  Here, our loss is based on the $L^2(\mcU,\mu)$ norm, and we consider all models as outlined in Assumption \ref{a:main}. 
A more delicate point is that we should also specify which quantities in the problem are ``fixed", that is, which quantities should be treated as fixed when constructing the scenario that achieves the desired lower bound.  This is important since our problem is regression and we are treating the predictors as fixed.  So consider $\mcM$ to be the collection of all possible distributions for $\{Y_{ ij \ell} \}$ for a fixed set of predictors $\{X_{ik}\}$ and fixed $m_i$ (though the $m_i$ are still allowed vary with $n$) satisfying Assumption \ref{a:main}.
We also assume that the parameters of interest lie in a closed bounded ball of $\mbK$: $\| \beta_{\ell p}\|_{\mbK} \leq M_0$, which will be denoted as $B_{\mbK}$.  So each $M \in \mcM$ indicates the distributions for $(\epsilon_{i\ell}, \delta_{ij\ell}, u_{ij})$ and specifies the values of $\beta_{\ell p}$.  

Define the excess risk:
\[
R_n= \sum_{p=1}^P \sum_{\ell=1}^L \| \hat \beta_{\ell p} - \beta_{\ell p}\|^2.
\]
We say that the rate of convergence of $\hat \beta$ is $a_n$ if $R_n = O_P(a_n)$.  The minimax estimation risk is then defined as the optimal rate of convergence (i.e. the smallest $a_n$), across all possible estimators, in the worst case modeling scenario.  
\begin{thm}\label{t:lower}
Let $\mcM$, as described above, be the collection of probability models satisfying Assumption \ref{a:main} with $\|\beta_{\ell p}\|_{\mbK} \leq M_0 < \infty$.  Then for any $\hat \beta$ which is a function of the data, the excess risk satisfies
\[
\limsup_{n \to \infty} \sup_{M \in \mcM}  P( R_n \leq \epsilon ((nm)^{-2h/(2h+1)} + n^{-1} )) \to 0
\qquad \text{as } \epsilon \to 0,
\]
if the arithmetic and harmonic means of the $m_i$ are asymptotically equivalent and the eigenvalues, $\tau_k$, of $K$, decay as $\tau_k \asymp k^{-2h}$. 
\end{thm}

The proof of Theorem \ref{t:lower} is given in the appendix.  It shows that no estimator can achieve a ``worst case" rate faster than $(nm)^{-2h/(2h+1)} + n^{-1}$; we will show in the next section that this bound is tight by giving an estimator that achieves the lower bound.  The proof is based on an application of Fano's lemma.  We show that a sequence of parameters within the ball $B_{\mbK}$ can be selected which are sufficiently far apart with respect to the $\mbK$-norm. We then prove a bound on the Kullback-Leibler divergence between any pair of probability measures induced by this collection of parameters. Combining these two bounds, we are able to apply Fano's lemma to obtain the desired result.   

One interesting caveat to Theorem \ref{t:lower} is the requirement that the arithmetic and harmonic means of the $m_i$ be asymptotically equivalent.  This is not simply a theoretical convenience as a case where this doesn't hold becomes surprisingly delicate. For example, suppose that one of the $m_i$ was essentially infinite (implying the entire curve is observed).  Then the arithmetic mean would be infinite, but the convergence rate need not be parametric.  Alternatively, if even of a small fraction of the $m_i$ were infinite (or very large), then the rate would become parametric, however the harmonic mean need not be infinite especially if the remaining $m_i$ are small.  If one were to let the fraction of $m_i$ being infinite (or very large) change with $n$, then one could obtain basically any convergence rate desired (between nonparametric and parametric), all while maintaining a bounded harmonic mean and an infinite arithmetic mean.  To avoid this, the lower bound given in \cite{cai2012minimax} was also taken over all $m_i$ that satisfy a specific harmonic mean, however this is somewhat strange since the $m_i$ are actually observed in a given problem.  Recently \cite{zhang2018optimal} discussed choosing optimal weights in place of $1$ or $1/m_i$ in the context of mean and covariance function estimation.  However, the weights were chosen to optimize the asymptotic upper bound of a local linear smoother and depended on the choice of the smoothing parameter.  

\subsection{Upper Bound} \label{s:up}
Recall that our proposed estimator is given by
\begin{align}
\hat \bbeta_l = (\bT_{nm} + \lambda \bI)^{-1} \bh_{nml}. \label{e:estimate}
\end{align}
We first give a more general result that provides a deeper understanding of the components of the convergence rate.  
\begin{thm}\label{t:upper}
Assume that \ref{a:main} holds and that $\hat \bbeta$ is as given in \eqref{e:estimate}.  If $\lambda$ is such that $nm \lambda^{\delta +  1/2h} \to \infty $ for some $\delta > 1/2h$, then the excess risk satisfies
\[
R_n = O_P(1) \left[ \lambda  + \frac{1}{\lambda^{1/2h} nm} + \frac{1}{n} 
\right].
\]
\end{thm}
Here we see three core components driving the statistical properies of $\hat \bbeta$.  As is common in nonparametric smoothing, the bias is given by $\lambda$.  The stochastic error is driven by two components.  The first is driven by the total number of observed values and takes a familiar ``nonparametric rate." The last component is a parametric rate, but only decreases with $n$, reflecting that there is a bounded amount of information that can be extracted from a single function/unit.  Balancing the bias and stochastic error, we arrive at the optimal rate of convergence.  

\begin{thm}\label{t:upper2}
Assume that \ref{a:main} holds and that $\hat \bbeta$ is as given in \eqref{e:estimate}.  If $\lambda \asymp (nm)^{-2h/(1+2h)} $ then the excess risk satisfies
\[
\limsup_{n \to \infty}  \sup_{\beta_{\ell k \in B_{\mbK}}} P( R_n \geq \epsilon^{-1} ((nm)^{-2h/(2h+1)} + n^{-1} )) \to 0
\qquad \text{as } \epsilon \to 0.
\]
\end{thm}

Combining Theorems \ref{t:lower} and \ref{t:upper2} we get that the minimax rate of converges is $(nm)^{-2h/(2h+1)} + n^{-1}$.  Furthermore, this rate holds quite broadly across different $\mbK$.  
The \textit{phase-transition} occurs when the rate becomes parametric, i.e., $n^{-1}$.  Clearly this occurs if
\[
(nm)^{-2h/(2h+1)} \ll n^{-1}
\Longleftrightarrow m \gg n^{1/2h}.
\]
In other words, the rate becomes parametric if the (harmonic) average number of points per curve is more than $n^{1/2h}$.  If $m$ is less, then the rate is slower than parametric.  In the worst case, when $m$ is bounded, the rate becomes the classic nonparametric rate $n^{-2h/(2h+1)}$.

\subsection{Interpreting the rate}\label{s:hrate}
In our theory, $h$ is only tied to the rate of decay of the eigenvalues of the RKHS kernel.  However, there are settings where this rate can be made more interpretable.  In the remainder of this section, we state the following theorem for Riemannian manifolds, which ties together several classic results from nonlinear analysis, and extends well-known connections between RKHS and Sobolev spaces for Euclidean spaces.  As the proof uses a number of results that might be of interest to readers, we state it here instead of in the appendix.  

\begin{thm}\label{t:hrate}
Let $\mcU$ be a compact $d$-dimensional Riemannian manifold.  Let $H^r(\mcU)$ denote the Sobolev space of real valued functions whose first $r$ weak derivatives are in $L^2(\mcU)$ and assume $2r > d$. Then $H^r(\mcU)$ is a reproducing kernel Hilbert space and the eigenvalues of the reproducing kernel decay like $\tau_k \asymp k^{-2 r/d}$.
\end{thm}
\begin{proof}
The Sobolev space, $\mcH^r:=\mcH^r(\mcU)$, of real functions over $\mcU$ with $r$ weak derivatives in $L^2(\mcU)$ can be continuously embedded into the space of continuous functions, $C(\mcU)$, if $2r > d$ \citep[][e.g. Section 2.3]{hebey2000nonlinear}.  This means that we can identify each $f \in H^r$ as the unique continuous representative of its corresponding equivalence class.  Then $\|f\|_{C(\mcU)} \leq M \|f\|_{\mcH^r} $, for some $M>0$ (across all $f$).  
Since $f(x) \leq \|f\|_{C(\mcU)}$ this means point-wise evaluation would be a continuous linear functional on $\mcH^r$  
and by the Riesz representation theorem, the space must also be an RKHS (recall a Hilbert space where point-wise evaluations are continuous is necessarily an RKHS).  

One can construct a kernel function that gives rise to $\mcH^r$ using the Laplace-Beltrami operator (i.e. the Laplacian for manifolds) acting over the space of infinitely differentiable functions, $\Delta: C^\infty(\mcU) \to C^\infty(\mcU)$.  This operator has eigenvalues tending to infinity, which we will label $0 \leq \xi_1 \leq \xi_2 \leq \dots $ (and can be zero), and corresponding eigenfunctions $v_1,v_2,\dots$, which, while infinitely differentiable, can be taken to be an orthonormal basis of $L^2(\mcU)$ \citep[Section 7.1]{canzani2013analysis}.  The Sobolev space $\mcH^r$ can be identified as
\[
\mcH^r = \left\{f \in L^2(\mcU) : \sum_{k=1}^\infty \xi_k^{r} \langle v_k, f \rangle^2 < \infty \right\},
\]
see, e.g., Chapter 3 of \cite{craioveanu2013old}.
Note that the first eigenvalue, $\xi_k^r$ is usually zero, meaning we do not restrict a function $f$ in that direction.  
We can equip $\mcH^r$ with a norm equivalent to the Sobolev norm as
\[
\|f\|^2_{\mcH^r}:= \sum_{k \leq k_0} \langle f, v_k\rangle^2 + \sum_{k>k_0} \xi_k^{r} \langle f , v_k \rangle^2,
\]
where $k_0$ is any integer satisfying $\xi_k > 0$ for $k>k_0$, thus avoiding the zero eigenvalue (taking $k_0 = 0$ would only result in a semi-norm, not a norm). 

Weyl's law for compact Riemannian manifolds \citep[Section 7.8]{canzani2013analysis} implies that $\xi_k \asymp k^{2/d}$ resuling in $\tau_k \asymp k^{-2 r/d}$.
\textcolor{black}{
Now define the linear operator
\[
K:= \sum_{k \leq k_0} v_k \otimes v_k + \sum_{k > k_0} \xi_k^{-r} v_k \otimes v_k,
\]
where the role of $\tau_k$ is now taken by either $1$ or $\xi_k^{-r}$.  Since we assume that $2h/d > 1$, this implies that $K$ is actually a Hilbert-Schmidt operator acting on $L^2(\mcU)$ (in fact it is trace class) and thus it must also be an integral operator and we can use its kernel as the reproducing kernel of the space.} 

%
%
\end{proof}

According to Theorem \ref{t:hrate}, we have that $h = r/d$ for Sobolev spaces over domains represented as compact Riemannian manifolds (this connection was already known for Euclidean spaces).  The minimax rate and phase transition become
\[
(nm)^{\frac{-2r}{2r+d}} + n^{-1} \qquad \text{and} \qquad
m \asymp n ^{d/2r}.
\]
We can see the effect of the dimension of the domain on the rates.  As we move to higher dimensions the rates get worse, while they improve if the parameters have more derivatives.  However, the key point to note is that the rate depends only on the intrinsic dimension of the manifold and not on the dimension of any ambient space. The point where one hits a parametric rate, which is commonly used to distinguish between dense and sparse functional data \citep{zhang2016sparse}, is much higher for more complex domains.  For example, it is common to assume $r=2$ derivatives in practice.  For one dimensional domains, the phase transition would then occur at $m \asymp n^{1/4}$, which is a relatively easy threshold to meet, while for two dimensions it becomes $n^{1/2}$ and over three it becomes $n^{3/4}$, meaning that one needs nearly as many points per curve as one has subjects, which is a much more stringent threshold.


\section{Numerical Illustrations} \label{s:num}
In this section we provide a simulation study to numerically explore the estimation error and also provide an application to 3D facial imaging data. Before providing the simulation results and data application, we briefly describe how our estimators are computed.

\subsection{Computation}\label{s:comp}
Using the representer theorem one can obtain an exact expression for the estimator.  However, this turns out to be very inefficient computationally as it involves solving for $\sum_i m_i$ parameters.  Instead, we will approximate the estimator for $\beta_p$ using the first $k_0$ eigenfunctions of $K(u, u')$:
\[
\beta_{p k_0}(u) = \sum_{k=1}^{k_0} b_{pk} v_k(u).
\]
We provide an exact form for the the coefficients $\{b_{pk} \}$.  As long as $k_0$ is chosen large enough, then the truncation error will be of a lower order than the convergence rate.  If $\beta$ all lie in a $\mbK$ ball then the truncation error is of the order
\[
\| \beta_p - \beta_{p k_0}\|^2
= \sum_{k=k_0+1}^\infty b_{pk}^2 = \sum_{k=k_0+1}^\infty \tau_k \frac{b_{pk}^2}{\tau_k}\leq \tau_{k_0} \|\beta_p\|_{\mbK}^2 \asymp k_0^{-2h}. 
\]
We see that as long as $k_0 \gg n^{1/2h}$ and $k_0 \gg (nm)^{1/(2h+1)}$ then the truncation error will be asymptotically negligible.  Of course, in practice, one can take $k_0$ much larger as long as the computational resources allow.

For simplicity, we assume that $m_i \equiv m$, but the general case can be handled by reweighting the $X_{ip}$ and $Y_{ij}$ and using $\bar{m} = \frac{1}{n} \sum_{i=1}^n m_i$ in place of $m$. Let $\bb = \{ b_{pk} \} \in \mathbb{R}^{P \times k_0}$. The target function is now given by
\begin{equation} \label{e:ls_obj}
\ell_{nm,\lambda}(\bb) =
\frac{1}{nm} \sum_{i=1}^n \sum_{j=1}^{m} \left(Y_{ij} - 
\sum_{p=1}^P \sum_{k=1}^{k_0} X_{ip} b_{pk} v_k(u_{ij})\right)^2 +  \lambda \sum_{p=1}^P \sum_{k=1}^{k_0} \frac{b_{pk}^2}{\tau_k}.
\end{equation}

We will rewrite this expression using vector/matrix notation.  First, let $b_v = vec(\bb)$, where $vec$ denote stacking the columns into a single vector.  Properties of the vec operation imply that
\[
\sum_{k=1}^{k_0} X_{ip} b_{pk} v_k(u_{ij}) 
= \bX_i^\top \bb V_{ij}
= ( V_{ij}^\top \otimes \bX_i^\top) b_v
\]
where $V_{ij} = \left(v_1(u_{ij}), \dots, v_{k_0} (u_{ij}) \right)^\top$.
Define 
\[
Y_v = \vc(\bY),
\qquad
\bA = \!\begin{aligned}[t]
      \Biggl\{
            &(V_{11}^\top \otimes \bX_1^\top),
            (V_{21}^\top \otimes \bX_2^\top),
            \cdots, 
            (V_{n1}^\top \otimes \bX_n^\top),\\
            &(V_{12}^\top \otimes \bX_1^\top),
            (V_{22}^\top \otimes \bX_2^\top),
            \cdots,
            (V_{n2}^\top \otimes \bX_n^\top),\\
            &\cdots,\\
            &(V_{1m}^\top \otimes \bX_1^\top),
            (V_{2m}^\top \otimes \bX_2^\top),
            \cdots,
            (V_{nm}^\top \otimes \bX_n^\top)
        \Biggr\}^\top
    \end{aligned}
\]
and let $T$ be a diagonal matrix with its diagonals corresponding to $\{\tau_k \}$, $k = 1, \dots, k_0$. 
Then the target function becomes
\[
\frac{1}{nm} (Y_v - \bA b_v)^\top (Y_v - \bA b_v) + b_v^\top (\Tau^{-1} \otimes \lambda I_P) b_v.
\]
The solution can then be expressed as
\[
\hat b_v = \left( (nm)^{-1} \bA^\top \bA + (\Tau^{-1} \otimes \lambda I_P)\right)^{-1} (nm)^{-1} \bA^\top Y_v.
\]
We choose the tuning parameter, $ \lambda $, using generalized cross-validation.  In the application we allow each $\beta_k$ to have a separate tuning parameter. If $\lambda_k$ is the tuning parameter for $\beta_k$, we can put $\Lambda$ instead of $\lambda I_P$ above where $\Lambda$ is a diagonal matrix with its diagonals corresponding to $\{ \lambda_p \}$, $p=1, \dots,P$. We cycle several times through each predictor selecting the best value. 

\subsection{Simulation}
\label{s:sim}
In this section, we evaluate the numerical performance of our estimator. We use a simplified setting with a one-dimensional outcome and one predictor to illustrate the effects of the domain $\mcU$, the sample size $n$, the number of observations per sample $m$, and different levels of smoothness of the underlying parameters.  The underlying model becomes
\[
Y_i(u) = X_i\beta(u) + \vep_i(u).
\]
For our simulation, we will use the Mat{\'e}rn kernel as it has parameters that directly controls the smoothness and its resulting RKHS can be tied to a particular Sobelev space \citep{aronszajn1961theory,cho2017compactly}. The Mat{\'e}rn kernel has the form 
\[
K(u,w) = \frac{2^{1-\nu}}{\Gamma(\nu)} \frac{\sqrt{2 \nu} \|u - w \|^2 }{\rho} K_\nu \left( \frac{\sqrt{2 \nu} \|u - w \|^2 }{\rho} \right)
\] 
where $K_v$ signifies the modified Bessel function of the second kind of order $\nu$. The smoothing parameter $\nu$ controls the smoothness of the resulting RKHS, and the range parameter $\rho$ scales the distance between $u$ and $s$. Larger $\nu$ would mean that the resulting RKHS will be smoother, and its eigenvalues will decay faster.

We generate the beta function for the $s$th simulation setting as
$$\beta^s (u) = \sum_{k=1}^{k_s} v^s_k(u) + \sum_{k > k_s} \tau^s_k  v^s_k(u)
$$
where $\{v^s_k\}$ are the eigenfunctions of Mat{\'e}rn kernel $K$ with smoothness parameter $\nu_s$ and range parameter $\rho = 1$. The number of leading eigenfunctions for the beta is $J_s$ while the eigenvalues and eigenfunctions of the RKHS are estimated using the algorithm of \cite{pazouki2011kernelbase}. We have created three settings for the dimension of $\mcU$ being one ($d=1$) and three settings for the dimension of $\mcU$ being two ($d=2$), and they are different in terms of the smoothness of RKHS where the beta lies ($\nu_s$) and the number of leading eigenfunctions ($k_s$). The resulting $\beta^s$ functions are shown in Figure \ref{fig:sim_beta}. The setting 1 beta function is the roughest and the setting 3 beta function is the smoothest for $d=1$ cases, and the setting 4 beta function is the roughest and the setting 6 beta function is the smoothest for $d=2$ cases.

\begin{multicols}{2}
{\small
\begin{itemize}
\item[] $d=1$
\begin{itemize}
\item Setting 1: $\nu_1 = 3/2$, $k_1 = 7$.
\item Setting 2: $\nu_2 = 7/2$, $k_2 = 5$.
\item Setting 3: $\nu_3 = 11/2$, $k_3 = 3$.
\end{itemize}
\item[] $d=2$
\begin{itemize}
\item Setting 4: $\nu_4 = 11/2$, $k_4 = 7$.
\item Setting 5: $\nu_5 = 15/2$, $k_5 = 5$.
\item Setting 6: $\nu_6 = 19/2$, $k_6 = 3$.
\end{itemize}
\end{itemize}
}
\end{multicols}

\begin{figure}[ht]
    \centering
    \includegraphics[width=.34\linewidth]{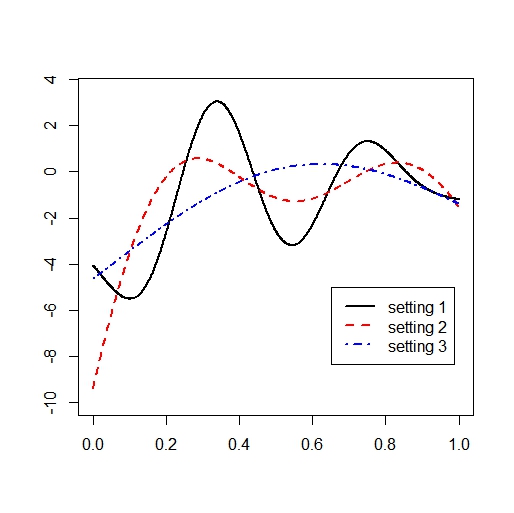}
    \includegraphics[width=.64\linewidth]{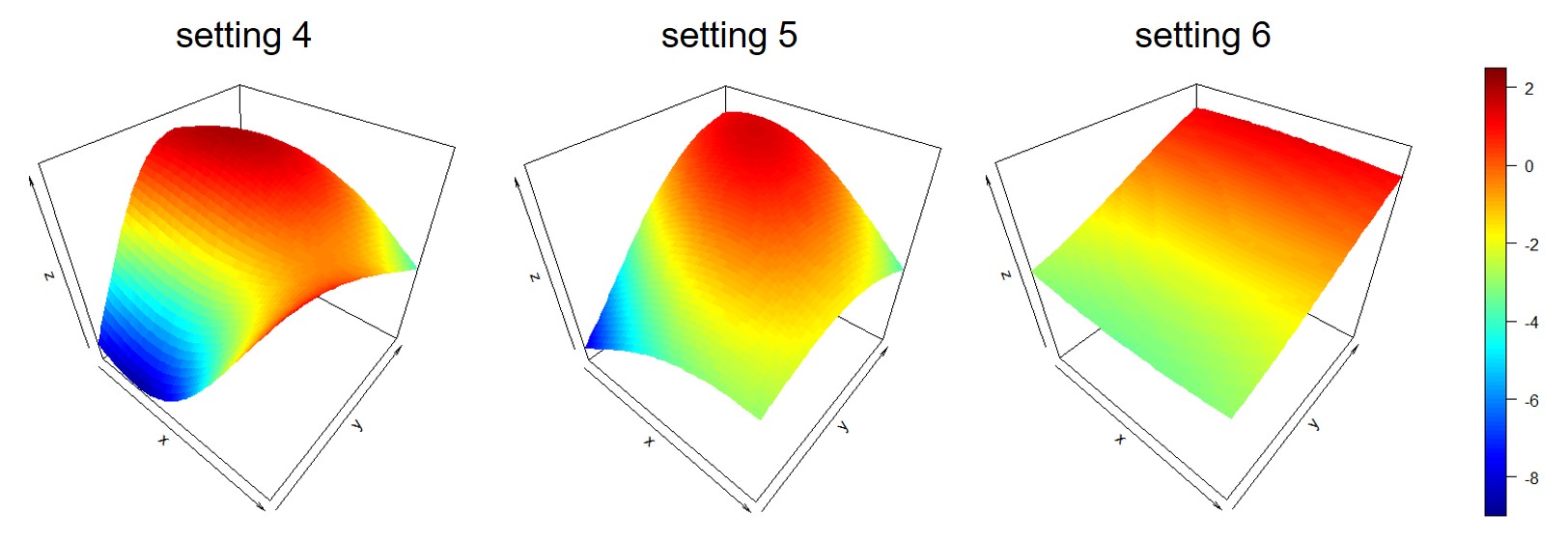}
    \caption{The beta functions generated for each simulation setting. The leftmost plot shows $d=1$ cases and the right three plots show $d=2$ cases.}
    \label{fig:sim_beta}
\end{figure}

We generate the observed data as
$$Y_{ij} = X_i \beta(u_j) + \epsilon_i (u_j) + \delta_{ij}$$
where $i = 1, \cdots, n$, $j = 1, \cdots, m$, $X_i \sim N(1,1)$, $\epsilon_i (u_j) = \sum_{} \epsilon_{ik} v_k(u_j)$ with $\epsilon_{ik} \sim N(0, \tau_k^2)$, and $\delta_{ij} \sim N(0, 0.1)$. For each simulation setting, we tried $n = 10, 25, 50, 75, 100, 125, 150$ and $m = 5, 10, 25, 50, 75, 100$ and ran 1000 repetitions of each scenario. We assume that we do not know the true kernel, but we will use Mat{\'e}rn kernel to estimate the beta. We choose the smoothness parameter $\nu$ and the length parameter $\rho$ for the RKHS where the estimated beta lies through the generalized cross validation (GCV). 
The choice of $\lambda$ in \eqref{e:ls_obj} is also done through GCV. 
Then we find the squared estimation error of $\| \hat{\beta}^s - \beta^s \|^2$ for each run and the mean squared estimation error of 1000 simulation runs for each $n$ and $m$ is shown in Figure \ref{fig:beta_err_1d} and Figure \ref{fig:beta_err_2d}.

\begin{figure}[ht]
    \centering
    \includegraphics[width=.99\linewidth]{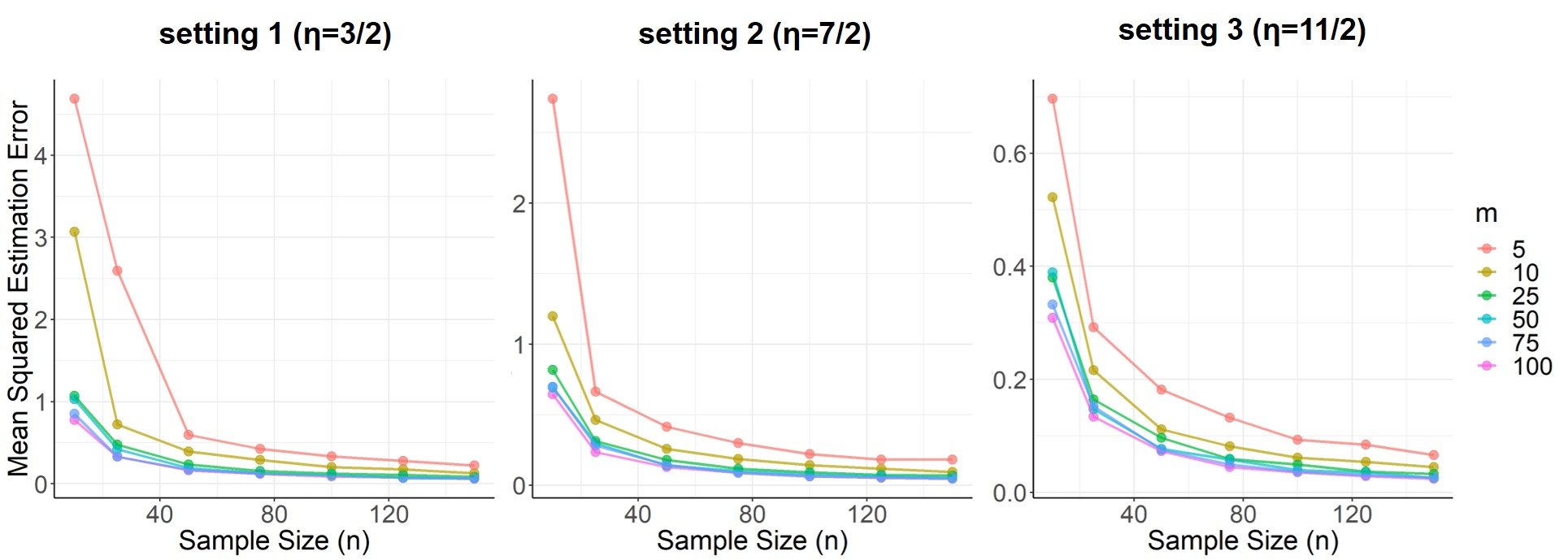}
    \caption{The effects of $n$ on the mean squeared estimation errors for $d=1$.}
    \label{fig:beta_err_1d}
\end{figure}

\begin{figure}[ht]
    \centering
    \includegraphics[width=.99\linewidth]{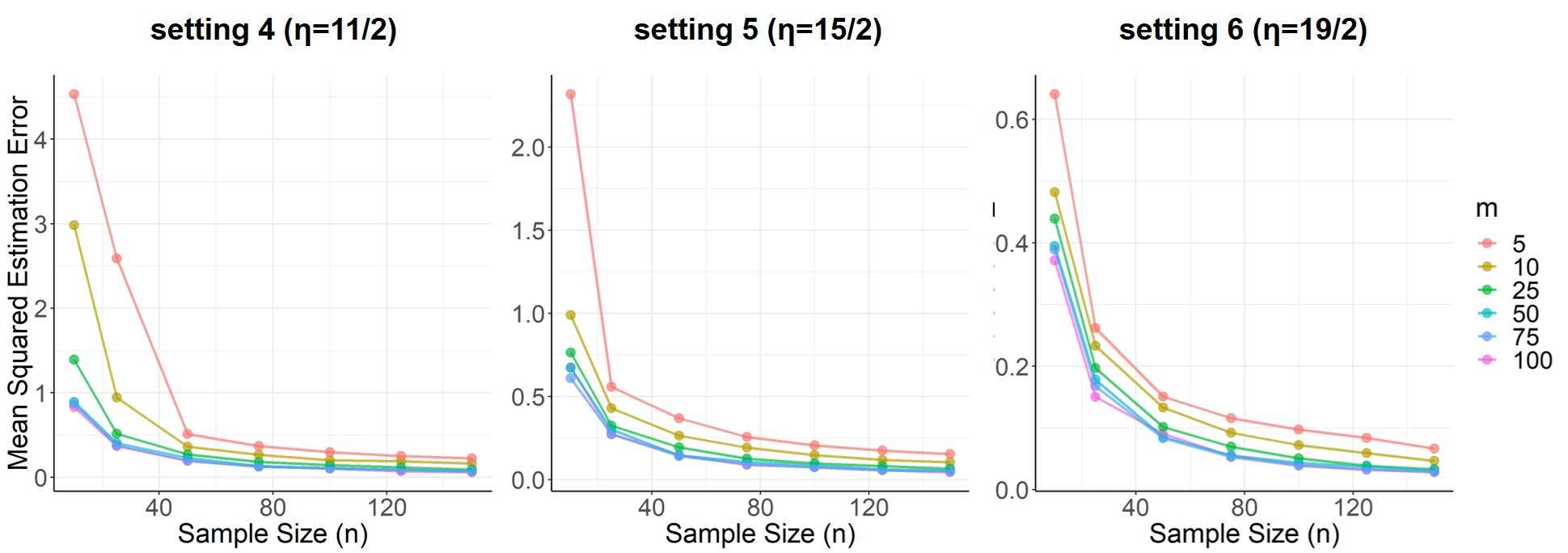}
    \caption{The effects of $n$ on the mean squeared estimation errors for $d=2$.}
    \label{fig:beta_err_2d}
\end{figure}

\noindent {\bf Discussion}: For both $d=1$ cases and $d=2$ cases, the mean squared estimation error decreases significantly as $n$ increases. We can also find the {\it phase-transition} that $n^{-1}$ becomes the dominating component of the convergence rate as $m$ is larger than $n^{\frac{1}{2h}}$, or $n^{\frac{d}{2r}}$, as the mean squared estimation error lines overlap with each other when $m$ is sufficiently large. For the same $n$ and $m$, the mean squared estimation error decreases as $\nu$ increases, or the true beta lies in the smoother RKHS. From setting 1 to setting 3, the mean squared estimation error decreases for the same $n$ and $m$, and the same happens from setting 4 to setting 6. When we compare Figure \ref{fig:beta_err_1d} and Figure \ref{fig:beta_err_2d}, one sees a fairly similar estimation error for both $d=2$ and $d=1$, however, the smoothness required to accomplish this is much higher in the two dimensional setting, which highlights the increased difficulty of the problem.  

\subsection{3D Facial Data} 
\label{s:adapt}
We apply our optimal estimator to the facial data collected through the Penn State ADAPT study \citep{claes:hill:shriver:2014,claes:etal:2014}. Following the framework of \cite{kang2017manifold} (who based their models on felsplines and FPCA), we fit a manifold-on-scalar regression model with the dependent/outcome variable being a 3D human facial face parametrized by a two-dimensional manifold $\mcU$ representing a common template face (we use the average face), and the independent/explanatory variables as sex, age, height, weight, and genetic ancestry. Genetic ancestry is measured as the proportions from particular ethnic backgrounds: Northern Europe, Southern Europe, East Asia, South Asia, Native America, and West Africa. We also include interactions between sex and age, age and weight, and height and weight.

The faces are densely measured with 7150 points in x, y, and z coordinates, so $Y_{ijl}$ in \eqref{e:model} will be the measurement of the $j$-th point of the $i$-th person's face in the $l$-th coordinate.  The sample size is $n = 3287$, with $m = 7150$ and $L = 3$. Since the template face, $\mcU$, is two-dimensional manifold, $d=2$. There are in total $P=13$ predictors including the intercept term. Prior to model fitting all faces are scaled and aligned using generalized procrustes analysis.  The computation follows section \ref{s:comp}, and the choices of $\lambda$, $\nu$, and $\rho$ are done through GCV. 

Four of the resulting $\hat{\beta}_{k}$'s, which are 3-dimensional functional objects, are shown in Figure \ref{fig:ADAPT_bhat}. The middle plot is the predicted face of a Northern European male, aged 30, with height of 170cm and weight of 70kg.  In each corner we repeat the prediction, but with one covariate value changed. On the top left is the predicted face of a Northern European female with the same age, height, and weight. The red and blue plot in between is the visualized estimated beta for sex. Red means there is a shift of the face outward, and blue means inward. From male to female, the red on the cheek and the blue on the chin show that the face becomes a bit rounder, and the red on the eyelids and the blue on the eyebrow give less prominent eyebrows and rounder eyes. Also, the slight hint of red around the nostrils show that female would also have a bit rounder nose. 

\begin{figure}[ht]
    \centering
    \includegraphics[width=.95\linewidth]{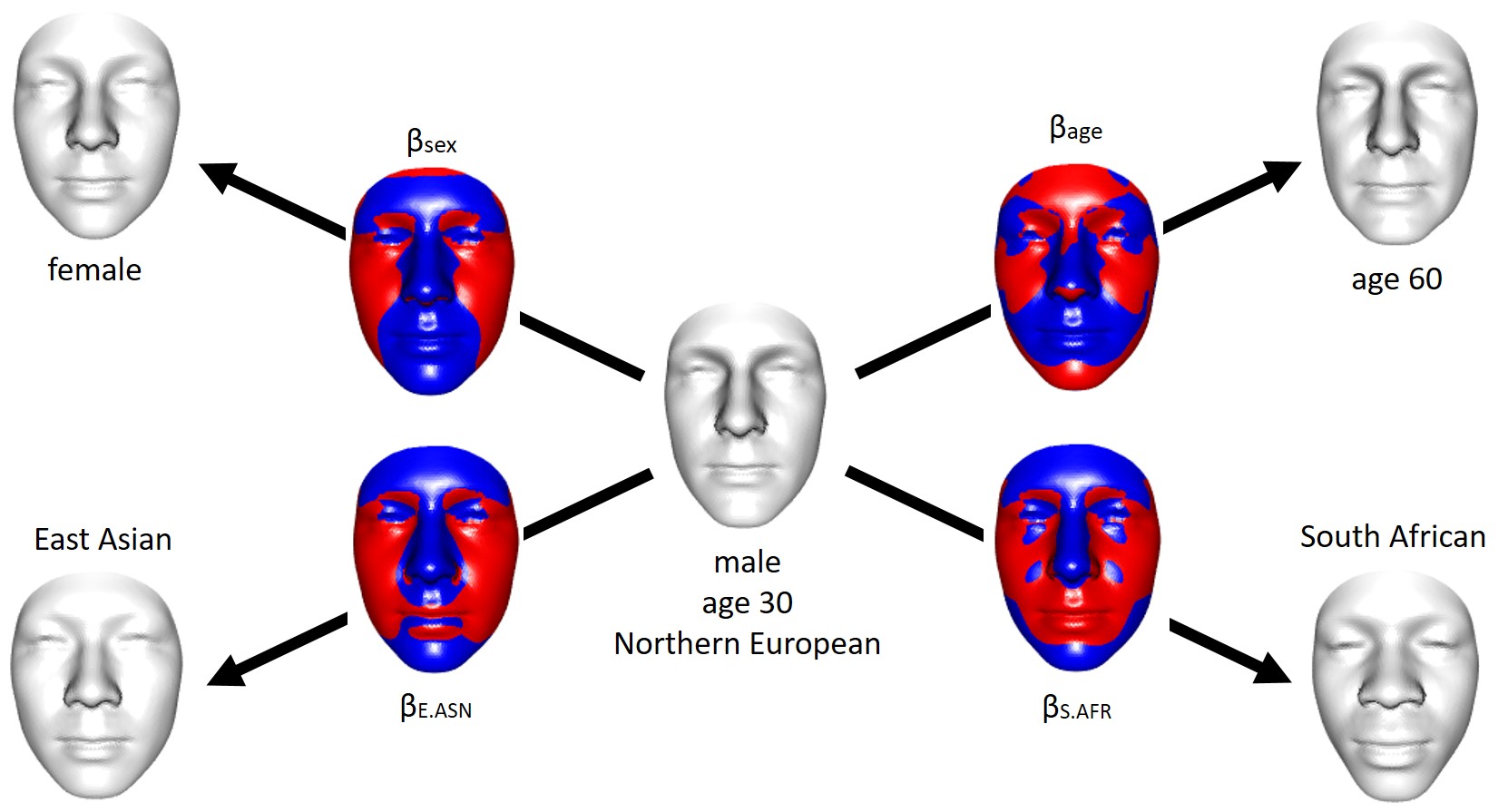}
    \caption{The middle grey plot is the predicted face of a Northern European male with age of 30, height of 170cm, and weight of 70kg. The four grey plots on the far sides are the predicted faces with one predictor change from the middle plot. The red and blue plots are the visualization of the effects of the corresponding estimated betas where the red means outward effect and the blue means inward effect. }
    \label{fig:ADAPT_bhat}
\end{figure}

On the top right is the predicted face when changing the age from 30 to 60 years old. The red and blue plot in between is again the visualized estimated beta for age. The red in the cheeks and jawline and the blue in between them show that the skin hangs more loosely on the cheek and jawline area, which (unfortunately) is a common aging effect. 
Another noticeable effect is on the eyes; the loose skin on eyelids and the bags under eyes are also well-known aging effects, and this is captured in the beta plot with the red on the eyelids and under the eye, and with the blue in the middle and on the sides of the eyes.

On the bottom left is the corresponding predicted face for East Asian ancestry, and the corresponding colored plot shows how the predicted faces differ between Northern Europeans and East Asians.
The red on the cheek and the blue on the chin shows that the predicted East Asian has a rounder face, and the blue on the nose with a little red on the sides mean that the predicted East Asian has a less prominent and slightly rounder nose than Northern European. Also, the predicted East Asian has less prominent eyebrows and forehead as the blue on that areas shows, and he has rounder eyes.

On the bottom right is the predicted face for South African ancestry. The plots indicate that the nose of the predicted South African is flatter and wider with the blue in the middle of nose and the red on the sides of nose. There seems to be minor tear-through nasojugal grooves under the eyes and the nasolabial folds below the nose in the predicted face of a South African, and these lines are captured with the blue dots on the cheeks. 

\section{Conclusions}

In this work we have presented new results concerning minimax rates for function-on-scalar regression when the domain of the functions is more complex than just an interval.  Assuming the parameters reside in an RKHS results in the rates being closely tied to the decay of the eigenvalues.  However, the rates in such cases, and thus the difficulty of the problem, can be somewhat hidden behind the eigenvalues.  To add clarity to our results, we extend well known connections between RKHS and Sobolev spaces to the case where the domain of the functions are compact Riemannian manifolds.  

A great deal of biomedical imaging data is being collected and analyzed in scientific studies.  As our technologies progress, such statistical methods will become increasingly important.  This is especially critical if statistical tools are to keep pace with more ``black box" machine learning methods.  Indeed, though the data is complicated, a major selling point of our methodology (and most statistical methods) is the ability to provide clear interpretations for the effects in our model, which scientists and practitioners will find useful.  

We provided a practical strategy for implementing our methods via a basis representation based on the RKHS kernel being used, which avoids some of the large matrix inversion problems inherent in using the representer theorem.  This approach scales nicely and provides a flexible tool that can be applied in a variety of settings so long as the RKHS kernel can be defined and computed.  However, we don't view this estimator as definitive and would be excited to see what insights other researchers have when choosing kernels and modelling strategies for different applications. 

\bibliographystyle{abbrvnat}
\bibliography{OptFSReg}


\clearpage
\appendix
\section{Proof of lower bound}
In the following we give an adaptation of the proof in \cite{cai2012minimax} for the case of general RKHS.  Interestingly, the lower bound is only tight if the harmonic and arithmetic mean are asymptotically equivalent, that is, they grow at the same order with $n$.  This stems from their upper bound being in terms of the harmonic mean, but the arguments for the lower bound lead to the arithmetic mean (if one does not assume the $m_i$ are identical).  
We also provide some extra details for the interested reader.  
To prove the lower bound result, we will employ Fano's lemma and construct an example that achieves the worst case rate.  

Recall that for lower bounds, we need only find one model $M \in \mcM$, that achieves the desired rate.  We can thus make any assumptions we like as long as it remains a valid model.  So, assume only that $\beta_{11} \equiv \beta$ is nonzero and all other $\beta_{k\ell}$ are zero.  Furthermore, assume that only the $\epsilon_{i1} \equiv \epsilon_i$ and $\delta_{ij 1} \equiv \delta_{ij}$ are nonzero.  This reduces the problem from the original $L$ to $L=1$.  So, without loss of generality we can let $L=1$, $P=1$.
Unfortunately, we cannot assume that $X_{i} \equiv 1$, since nowhere did we say that the intercept is always included in the model.   
We assume the radius of $B_\mbK$ is one since it does not play any role in the arguments. In this case the model is given by
\[
Y_i(u_{ij}) = X_i \beta(u_{ij}) + \vep_i(u_{ij}) + \delta_{ij}.
\]
Assume that the distribution of the observed locations $u_{ij}$ has a uniform density with respect to the base measure $\mu$.  Assume that $\vep_i$ are iid mean zero Gaussian processes with covariance function, $C(u,u')$, and that the $\delta_{ij}$ are iid mean zero normals with variance $1$.  
Assume that $C(u,u') \equiv 1$.  
Then each parameter in the RKHS ball, $\beta \in B_\mbK$, induces a Gaussian probability measure over $\mbR^{n m}$.  Consider parameters, $\beta_1,\dots, \beta_M \in B_\mbK$ and their induced probability measures $P_1,\dots,P_M$.  Fano's lemma tells us the following.  
\begin{lemma}[Fano's Lemma]
Let $P_1,\dots,P_M$ be probability measures over $\mbR^{n  m}$ such that
\[
KL(P_i || P_j) \leq \alpha, \qquad i \neq j
\]
then for any test function $\psi:\mbR^{nm} \to \{1,\dots,M\}$ we have
\[
P_i(\psi = i) \leq \frac{\alpha + \log 2}{\log(M-1)}
\qquad \text{or} \qquad
P_i(\psi \neq i) \geq 1- \frac{\alpha + \log 2}{\log(M-1)}.
\]
\end{lemma}
In other words, Fano's lemma gives us an upper bound on the estimation accuracy for any possible test we could construct to select the true $\beta$ from among the $\beta_1,\dots,\beta_M$.  Any estimator, $\hat \beta$, we could construct in this setting would be equivalent to choosing one of the $\beta_1,\dots,\beta_M$.  Thus, in this case the estimation error must be at least
\[
\E_{P_i}\|\hat \beta - \beta_i\|^2
\geq P_i(\hat \beta \neq \beta_i) \min_{i,j}\| \beta_i - \beta_j\|^2
\geq \left(1- \frac{\alpha + \log 2}{\log(M-1)} \right) \min_{i,j}\| \beta_i - \beta_j\|^2
\]
for any estimator.  So, applying Fano's lemma becomes a matter of selecting $\beta_1,\dots,\beta_M$ that are well separated while properly balancing the KL divergence.  

First turning to the KL divergence, recall that between two Gaussian random vectors, $N(\mu_1,\Sigma)$ and $N(\mu_2,\Sigma)$, it is given by $(1/2)(\mu_1 - \mu_2)\Sigma^{-1}(\mu_1 - \mu_2).$  In addition each $P_i$ is composed of $n$ indepdent Gaussian measures over the product space $\mbR^{m_1} \times\dots \times \mbR^{m_n}$, over which the KL divergence is additive.  Let $\bu_i \sim (u_{i1},\dots,u_{im_i})$ and denote $\beta_j(\bu_i) = (\beta_j(u_{1i}),\dots,\beta_j(u_{jm_i}))^\top$, and $\Sigma(\bu_i) =\{C(u_{ij},u_{ik}) + \sigma_0^2 1_{j=k}\}$.  By assumption we can write
\[
\Sigma(\bu_i) = \bI_{m_i} + 1_{m_i} 1_{m_i}^\top,
\]
where $\bI_{m_i}$ is the $m_i$ dimensional identity matrix and $1_{m_i}$ is a vector of ones of length $m_i$.  Using the {\it Sherman-Morris formula} from linear algebra, we have
\[
\Sigma(\bu_i)^{-1} = \bI_{m_i}  - \frac{1}{m_i+1} 1_{m_i} 1_{m_i}^\top.
\]
Turning the KL divergence we have
\begin{align*}
KL(P_i||P_j)
& = \frac{1}{2} \sum_{k=1}^n X_k^2 \E[(\beta_i(\bu_k) - \beta_j(\bu_k))^\top \Sigma(\bu_k)^{-1}(\beta_i(\bu_k) - \beta_j(\bu_k))]
\end{align*}
From the Sherman-Morris formula, the above can be expressed as the sum of two components.  The first is
\begin{align*}
    \E[(\beta_i(\bu_k) - \beta_j(\bu_k))^\top \bI_{m_k}(\beta_i(\bu_k) - \beta_j(\bu_k))]
    =  m_k \| \beta_i - \beta_j\|^2.
\end{align*}
While the second is given by
\begin{align*}
    &\E[((\beta_i(\bu_k) - \beta_j(\bu_k))^\top 1_{m_k} )^2] \\
    & = m_k \| \beta_i - \beta_j\|^2 + m_k(m_k-1) \left( \langle 1, \beta_i - \beta_j \rangle \right)^2
     \geq m_k \| \beta_i - \beta_j\|^2.
\end{align*}
Putting the two together, we get a bound on the KL divergence of the form
\begin{align*}
\frac{\|\beta_i - \beta_j\|^2}{2} \sum_{k=1}^n X_k^2 \left( m_k - \frac{m_k}{1+m_k} \right)
=\frac{\|\beta_i - \beta_j\|^2}{2} \sum_{k=1}^n X_k^2 \frac{m_k^2}{1+m_k}
\leq \frac{\|\beta_i - \beta_j\|^2 n m_a \zeta^2}{2},
\end{align*}
where $m_a$ is the arithmetic mean and from Assumption $\ref{a:main}$ $X_k^2 \leq \zeta^2$.  Our estimation error is then bounded from below by
\[
\left( 1 - \frac{(\zeta^2/2) n m_a  \max_{ij}\|\beta_i - \beta_j\|^2 + \log(2)}{\log(M-1)} \right) \min_{ij}\|\beta_i - \beta_j\|^2.
\]
We want to make this error as large as possible (since that would produce the tightest lower bound), under the constraint that $\beta_i \in B_\mbK$.  
To construct a viable sequence, we consider the Varshamov-Gilbert bound \citep{varshamov1957estimate,duchi2016lecture}.
\begin{lemma}[Varshamov-Gilbert]
For $N \geq 1$ there exists at least $M = \exp(N/8)$ $N$-dimensional vectors, $b_1,\dots,b_M$, with entries $b_{ik} \in \{0,1\}$ such that
\[
\sum_{k=1}^N 1\{b_{ik} \neq b_{jk}\} \geq N/4.
\]
\end{lemma}
This is a commonly used lemma for constructing collections of parameters for minimax results as they take a very simple form.  We can use these sequences to construct elements of $L^2(\mcU)$ in the $v_k$ basis.  A suitable choice turns out to be the following:
\[
\beta_i(u) := N^{-1/2} \sum_{k=N + 1}^{2N} \tau_k^{1/2} b_{i,k - N} v_k(u).
\]
We clearly have the following properties
\begin{align*}
 \| \beta_i\|_\mbK^2 \leq 1, \qquad
 \| \beta_i - \beta_j\|^2 \geq  \tau_{2N} /4, \qquad 
 \| \beta_i - \beta_j\|^2 \leq \tau_{N}.
\end{align*}
Using this sequence, the lower bound becomes
\begin{align*}
& \left( 1 - \frac{(\zeta^2 /2) n m_a  \tau_{N} + \log(2)}{N/8} \right) \tau_{2N} 
 \asymp
\left( 1 - \frac{ 4 \zeta^2 n m_a N^{-2h} + 8 \log(2)}{N} \right) (2N)^{-2h}.
\end{align*}
Taking $N = (8 \zeta^2 n m_a)^{1/(1+2h)}$, which implies $N \to \infty$, would produce
\[
\left(\frac{1}{2} - \frac{8 \log(2)}{N} \right) (2N)^{-2h} 
 \asymp (n m_a)^{-2h/(1+2h)},
\]
which is the desired bound as long as $m_a \asymp m$.  
This bound (as we will see) matches the upper bound in the case where $m \asymp m_a$ and $m \ll n^{1/2h}$ or is of the same order, giving a tight rate. However, in the case where $m \gg n^{1/2h}$, then the bound is loose.

To obtain a bound that works when $m \gg n^{1/2h}$ we can make the problem even simpler. Assume that $\beta_j(u) \equiv b_j$, meaning that there are no dynamics in time (one just has a repeated measures problem).  A simple verification shows that the vector of all ones is an eigenvector with eigenvalue $m_i+1$. This implies that the KL divergence is now bounded by
\[
KL(P_i || P_j) \leq \frac{(b_i - b_j)^2 }{2} \sum_{k=1}^n \frac{m_k X_i^2}{m_k + 1}
\leq \frac{n \zeta^2 (b_i - b_j)^2}{2}.
\] 
Now we actually only need four test values in this case.  Take $b_i = \delta (i-1)/(3\sqrt n)$ for $i=1,2,3,4$ and let $\delta^2/2 =  (3/4) \log(3) - \log(2)$.  Then we can bound the KL divergence as $\delta^2/2$, the resulting lower bound on the estimation error is
\[
\left( 1 - \frac{\delta^2/2 + \log(2)}{\log(3)}\right) \frac{\delta^2}{9 n} = 
\frac{\delta^2}{36 n} \asymp n^{-1},
\]
as desired.

\section{Proof of upper bound}	

Since each coordinate of the response can be estimated separately, we will assume wlog that $L=1$ in our proof.  We also assume, wlog, that $u_{ij}$ have a density identically equal to 1, meaning their law is given by $\mu$.  
 We assume the kernel $K(u,s)$ is continuous over $\mcU$, which means it is also bounded since $\mcU$ is compact.  Using Mercer's theorem it admits the spectral decomposition
\begin{align}
K(u,s) = \sum_{k=1}^\infty \tau_k v_k(u) v_k(s). \label{e:eigen}
\end{align}
We assume that eigenvalues decay as
\[
\tau_k \asymp k^{-2h},  
\]
for some $h \geq 1$. 
Recall that, by Mercer's theorem, the convergence above occurs uniformly and absolutely in $u$ and $s$.  We therefore have the following lemma, which will be used throughout.
\begin{lemma} \label{l:erate}
If $K(u,s)$ is a continuous, positive definite, and symmetric kernel then it admits the eigen-decomposition \eqref{e:eigen}, which satisfies
\[
\sup_{t,s}\tau_k |v_k(u) v_k(s)| \to 0 \quad \text{as k} \to \infty.
\]
\end{lemma}
The use of this Lemma \ref{l:erate} is what allows us to relax the assumptions on the error process as compared to \cite{cai2012minimax}, as it allows us to avoid certain Cauchy-Schwarz inequalities involving the errors (note it also fixes one misapplication of the Cauchy-Schwarz they had in their proofs).  The functions $v_k(u)$ are normalized to have $L^2(\mcU)$ norm one (from here on we notationally drop the domain $\mcU$), which also means they have $\mbK$ norm $\tau_k^{-1/2}$.  Recall that the $\mbK$ inner product can be expressed as
\[
\langle g, f\rangle_{\mbK} = \sum_{k=1}^\infty \frac{\langle f, v_k \rangle \langle g, v_k \rangle}{\tau_k},
\]
where norms and inner products without subscripts will always denote the $L^2$ norm.

We now define the biased population parameter that will act as an intermediate value in our asymptotic derivation.  First, define the population counterpart to $\bT_{nm}$ from Section \ref{s:up} as
\[
[\bT \bff](u) := \E[ K_{u_{11}}(u)  \bSig_X \bff(u_{11})] = \int K(u, s)  \bSig_X \bff(s) \ d \mu(s)
\]
and $\bh = \bT (\bbeta_0)$. 
We then define
\begin{align}
\bbeta_\lambda = (\bT + \lambda \bI)^{-1} \bh = (\bT + \lambda I)^{-1} \bT \bbeta_0. \label{e:glam}
\end{align}

We now define a final intermediate value as
\begin{align}
\tilde \bbeta_\lambda 
= \bbeta_\lambda + (\bT+\lambda \bI)^{-1}(\bh_{nm} - \bT_{nm}(\bbeta_\lambda)  - \lambda \bbeta_\lambda). \label{e:gtilde}
\end{align}
To establish our convergence rates we break up the problem into three pieces:
\[
\hat \bbeta - \bbeta_0 =(\bbeta_\lambda -\bbeta_0) + (\tilde \bbeta_\lambda - \bbeta_\lambda) +  (\hat \bbeta - \tilde \bbeta_\lambda)  .
\]
In order to establish bounds for the third term above, it will be necessary to bound the second term in terms of the norm $\| f \|_\alpha = \langle K^{-\alpha/2} f, K^{-\alpha/2} f \rangle$.  When $\alpha = 0$ this is the $L^2$ norm, when $\alpha = 1$ it is the $\mbK$ norm, but we allow intermediate values $\alpha \in [0,1]$.

\subsection*{Step 1: $\bbeta_\lambda - \bbeta_0$}
Using \eqref{e:glam} we have
\[
\bbeta_\lambda - \bbeta_0 = [(\bT+\lambda \bI )^{-1}\bT - \bI] \bbeta_0 = -\lambda  (\bT+\lambda \bI)^{-1}  \bbeta_0.
\]
We want to compute the norm of this quantity in the product space $(L^2)^P$, which, equivalently, can be thought of as the tensor product space $\mbR^P \otimes L^2$.  We can make this calculation cleaner by using an appropriate basis.  In particular, recall that $v_k$ are the eigenfunctions of $K$, and we can add to them the eigenvectors of $\Sigma_X$, denoted as $\bu_p$, we can then construct a basis for the space as
\[
\{\bu_p \otimes v_k: p=1,\dots,P \ k=1,\dots, \infty\}.
\]
If we let $ \eta_p$ denote eigenvalues of $ \Sigma_X$, then the eigenvalues of $(\bT+\lambda I)$ are $ \eta_p\tau_k + \lambda$ and the eigenfunctions are $\bu_p \otimes v_k$.  Applying Parceval's identity yields
\begin{align*}
\| \bbeta_\lambda - \bbeta_0\|^2  
& = \sum_{p=1}^P \sum_{k=1}^\infty \langle \lambda  (\bT+\lambda \bI)^{-1}  \bbeta_0, \bu_p \otimes v_k \rangle^2\\
&= \lambda^2 \sum_p \sum_k \frac{1}{  ( \eta_p \tau_k + \lambda )^2} \langle \bbeta_0, \bu_p \otimes v_k \rangle^2 \\
&=  \lambda^2 \sum_p \sum_k \frac{\tau_k}{( \eta_p \tau_k + \lambda )^2}  \frac{\langle \bbeta_0, \bu_p \otimes v_k \rangle^2}{\tau_k} \\
&\leq \lambda^2  \| \bbeta_0\|^2_\mbK \sup_{pk} \frac{\tau_k}{(\eta_p \tau_k + \lambda)^2}
\leq \lambda^2 \nu \| \bbeta_0\|^2_\mbK \sup_{pk} \frac{\eta_p \tau_k}{(\eta_p \tau_k + \lambda)^2}.
\end{align*}
To bound the sup consider the function $f(x)=x^\gamma(x+\lambda)^{-2}$, over $x \geq 0$ and for some fixed $\gamma > 0$ (this level of generality will be useful later on).  Notice that this function will attain its maximum at a finite value of $x$ if and only if $\gamma < 2$, for $\gamma \geq 2$ the maximum is attained at infinity.  The derivative is given by 
\[
\gamma x^{\gamma-1}(x+\lambda)^{-2}-2x^\gamma(\lambda+x)^{-3}.
\]
Setting equal to zero we have
\[
\gamma (\lambda+x)-2x=0 \Longrightarrow x =  \frac{\gamma}{2-\gamma}\lambda .
\]
So we have 
\begin{align}
\sup \frac{(\eta_p\tau_k)^{\gamma}}{(\eta_p\tau_k + \lambda)^2}
\leq c_0 \lambda^{\gamma-2}. \label{e:eigbound}
\end{align}
Note that throughout we take $c_0, c_1$, etc, to denote generic constants whose exact values may change depending on the context.
Taking $\gamma = 1$ we conclude that
\begin{align}
\| \bbeta_\lambda - \bbeta_0\|^2 \leq c_0 \lambda \nu  \| \bbeta_0\|^2. \label{e:bias}
\end{align}

\subsection*{Step 2: $\tilde \bbeta_\lambda - \bbeta_\lambda$}
In this part we will bound the difference more generally using the $\alpha$ norm for $\alpha < 1-1/2h$. 
First, recall that, by definition of $\bbeta_\lambda$ we have
\[
\bT \bbeta_\lambda + \lambda  \bbeta_\lambda = \bh \Longrightarrow
\lambda \bbeta_\lambda = \bh - \bT \bbeta_\lambda = \bT (\bbeta_0 -  \bbeta_\lambda).
\]
Plugging this into \eqref{e:gtilde}, the expression for $\tilde \bbeta_\lambda$, we obtain
\[
\tilde \bbeta_\lambda - \bbeta_\lambda
= (\bT + \lambda \bI)^{-1} \left[ \bh_{nm} - \bT_{nm} \bbeta_\lambda - (\bT \bbeta_0 - \bT \bbeta_\lambda)
\right].
\]
This quantity has mean zero since, using \eqref{e:hnm} we have 
\[
\E[\bh_{nm}](u) 
= \frac{1}{n} \sum_i \frac{1}{m_i} \sum_j
\bX_i \bX_i^\top \E[\bbeta(u_{ij})  K_{u_{ij}}(u)]
=(\bT \bbeta_0)(u).
\]
and using \eqref{e:Tnm} we have
\[
\E[\bT_{nm}\bbeta_\lambda](u) 
= (\bT \bbeta_\lambda)(u).
\]
Using Parceval's identity we can express the expected difference in the $\alpha$ norm as
\[
\E\|\tilde \bbeta_\lambda - \bbeta_\lambda\|^2_\alpha
= \sum_p \sum_k  \frac{1}{\tau_k^{\alpha} (\eta_p \tau_k + \lambda )^2 } \Var(\langle \bh_{nm} - \bT_{nm} \bbeta_\lambda, \bu_p \otimes v_k \rangle). 
\]
Using the assumed independence across $i$ and the definitions \eqref{e:hnm} and \eqref{e:Tnm} we have
\begin{align*}
 & \Var(\langle \bh_{nm} - \bT_{nm} \bbeta_\lambda, \bu_p \otimes v_k \rangle)   \\
  & =  \frac{1}{n^2} \sum_i \frac{1}{m_i^2}\Var \left(
\sum_\ell (Y_{i\ell} - \bX_i^\top \bbeta_\lambda(u_{i\ell})) \langle K_{u_{i\ell}}, v_k \rangle \bX_i^\top \bu_j \right).
\end{align*}
Using the reproducing property and that the $v_k$ are the eigenfunctions of $K$, we can express $\langle K_{u_{ij}}, v_k \rangle = \tau_k \langle K_{u_{ij}} , v_k \rangle_{\mbK} = \tau_k v_k(u_{ij})$.  So the above is bounded by
\begin{align*}
& \frac{\tau_k^2 }{n^2} \sum_i \frac{(\bX_i^\top \bu_j)^2}{m_i^2}\Var \left(
\sum_\ell (Y_{i\ell} - \bX_i^\top \bbeta_\lambda(u_{i\ell})) v_k(u_{i \ell}) \right) \\
&\leq 
\frac{\tau_k^2 P \zeta^2 }{n^2} \sum_i \frac{1}{m_i^2}\Var \left(
\sum_\ell (Y_{i\ell} - \bX_i^\top \bbeta_\lambda(u_{i\ell})) v_k(u_{i \ell}) \right).
\end{align*}

  Conditioning on the sigma algebra generated by the locations, $\mcF = \sigma\{u_{ij}\}$, we get
\begin{align*}
\Var \left(
\sum_j (Y_{ij} - X_i^\top \bbeta_\lambda(u_{ij}))  v_k(u_{ij})\right)
& = \Var \left( \E \left[
\sum_j (Y_{ij} - \bX_i^\top \bbeta_\lambda(u_{ij}))  v_k(u_{ij})
\biggr| \mcF \right] \right) \\
& + \E \left[ \Var \left(
\sum_j (Y_{ij} - \bX_i^\top \bbeta_\lambda(u_{ij}))  v_k(u_{ij})
\biggr| \mcF \right) \right].
\end{align*}
The first term is given by
\begin{align*}
\Var \left(
\sum_j \bX_i^\top (\bbeta_0(u_{ij}) - \bbeta_\lambda(u_{ij}))  v_k(u_{ij})
\right)
& = m_i \Var(\bX_i^\top (\bbeta_0(u_{11}) - \bbeta_\lambda(u_{11})  )v_k(u_{11}) ) \\
& \leq m_i  \E(\bX_i^\top (\bbeta_0(u_{11}) - \bbeta_\lambda(u_{11})) v_k(u_{11}) )^2 \\
& = m_i  \int [\bX_i^\top (\bbeta_0(u) - \bbeta_\lambda(u))]^2 v_k(u)^2 \ d \mu (u) \\
& \leq m_i |\bX_i|^2 \|\bbeta_0 - \bbeta_\lambda\|^2 \sup_u v_k(u)^2 \\
 & \leq c_0 P \zeta^2 m_i \tau_k^{-1} \lambda \| \bbeta_0 \|_{\mbK}^2.
\end{align*}
Note the last line follows from Lemma \ref{l:erate} and equation \eqref{e:bias}.

Turning to the second term, we have
\begin{align*}
\Var \left(
\sum_j (Y_{ij} - \bX_i^\top \bbeta_\lambda(u_{ij}))  v_k(u_{ij})
\biggr| \mcF  \right)
&= \sum_{j \ell} \Cov(Y_{ij}, Y_{i \ell} | \mcF) v_k(u_{ij}) v_k(u_{i \ell})\\
& = \sum_{j \ell} (C(u_{ij}, u_{i \ell}) + \sigma^2 1_{j=\ell} )v_k(u_{ij}) v_k(u_{i \ell}).
\end{align*}
When $j =  \ell$ we use the assumed bounded variance and the orthonormality of the $v_k$ to obtain
\begin{align*}
 \E[  (C(u_{ij}, u_{ij}) + \sigma^2 )v_k(u_{ij})^2] 
& =\int C(u,u) v_k(u)^2 \ du + \sigma^2
\leq c_0.
\end{align*}
When $j \neq \ell$ we use the definition of the covariance to obtain
\begin{align*}
\E[  (C(u_{ij}, u_{i \ell}) v_k(u_{ij} )  v_k(u_{i \ell}) ] 
& = \int \int  v_k(u)C(u,s) v_k(s)  \ ds du\nonumber\\
&= \langle v_k , C v_k \rangle 
= \E \langle \vep, v_k \rangle^2.
\end{align*}

Using generic $\{c_i\}$ for the constants and recalling that $m$ is the harmonic mean of the $m_i$ we get the bound 
\begin{align}
& \E \| \tilde \bbeta_\lambda - \bbeta_\lambda \|^2_\alpha \notag \\
&\leq 
\sum_p \sum_k \frac{ \tau_k^{2-\alpha}}{ (\eta_p \tau_k+ \lambda )^2} \frac{1}{n^2} \sum_{i=1}^{n} \frac{1}{m_i^2} 
\left[
\frac{c_0 m_i \lambda}{\tau_k} + m_i c_1 + m_i^2 \E \langle \vep , v_k \rangle^2
\right] \notag \\
& = \sum_p \sum_k \frac{ \tau_k^{2-\alpha}}{ (\eta_p \tau_k+ \lambda )^2} \frac{1}{n}\left[ \frac{\lambda }{m \tau_k} c_0 +    \frac{1}{m }c_1 +  \E \langle \vep, v_k \rangle^2
\right]. \label{e:bsum}
\end{align}
We bound each term in the summand separately.  If  $\tau_k \asymp k^{-2h}$ then so is $\eta_p \tau_k$, since $1 \leq p \leq P$.  For an arbitrary $\gamma > 1/2h$ we have 
\[
\sum_{k=1}^\infty \frac{ \tau_k^{\gamma}}{ ( \eta_p \tau_k + \lambda)^2}
\asymp \int_0^\infty \frac{x^{-2h\gamma}}{(\lambda + x^{-2h})^2} \ dx 
= \int \frac{x^{2h(2 - \gamma)}}{(\lambda x^{2h} + 1)^2} \ dx.
\]
Let $y = \lambda x^{2h}$ then $x = \lambda^{-1/2h} y^{1/2h}$ and $dx = \lambda^{-1/2h} (1/2h) y^{1/2h - 1} dy$.  Then the above becomes
\[
\int \frac{\lambda^{-(2-\gamma) }y^{2-\gamma}}{(y+1)^2}\lambda^{-1/2h} (1/2h) y^{1/2h - 1} dy
= \frac{\lambda^{-(2-\gamma+1/2h)}}{2h}\int \frac{y^{1-\gamma+1/2h}}{(y+1)^2} dy.
\]
Notice the integral is finite since $\gamma > 1/2h$.  We therefore have that, for any $\gamma > 1/2h$ and $p=1,\dots,P$, 
\begin{align}
\sum_{k=1}^\infty \frac{ \tau_k^{\gamma}}{ (\eta_p \tau_k + \lambda)^2}
\asymp \lambda^{-(2-\gamma+1/2h)}. \label{e:boundsum}
\end{align}

Taking $\gamma = 1 - \alpha$ and applying \eqref{e:boundsum}, which is  greater than $1/2h$  as long as $\alpha < 1 - 1/2h$, the first term in \eqref{e:bsum} is given by
\[
 \sum_{p=1}^P  \sum_{k=1}^\infty \frac{ \tau_k^{1-\alpha}}{ (\eta_p \tau_k+ \lambda )^2} \frac{\lambda c_0}{n m} = O( \lambda^{-\alpha - 1/2h} (nm)^{-1}).
\]
 Turning to the second term in \eqref{e:bsum}, take $\gamma = 2-\alpha$ we have by the same arguments
 \[
 \frac{c_2}{nm} \sum_p \sum_k \frac{\tau_k^{2-\alpha}}{(\eta_p \tau_k + \lambda)^2} \asymp (nm)^{-1} \lambda^{-\alpha - 1/2h}.
 \]
 Turning to the last term in \eqref{e:bsum} we can use that $\E \| \vep \|^2 < \infty$  to obtain
 \[
 \sum_p \sum_{k=1}^\infty \frac{ \tau_k^{2-\alpha}}{ (\eta_p \tau_k+ \lambda )^2} \frac{1}{n} \E \langle \vep, v_k \rangle^2
 \leq \E\| \vep\|^2 n^{-1} \nu^{2-\alpha} \max_k \frac{ (\eta_p \tau_k)^{2-\alpha}}{ (\tau_k+ \lambda )^2}. 
 \]
 Applying \eqref{e:eigbound} with $\gamma =2 - \alpha$ we have that the above is equivalent to 
  \[
 \E\| \vep\|^2 n^{-1} c_0 \lambda^{-\alpha},
 \]

 We thus conclude that
\[
\| \tilde \bbeta_\lambda - \bbeta_\lambda \|^2_\alpha = O_P \left( (nm)^{-1} \lambda^{-\alpha - 1/2h} + n^{-1} \lambda^{-\alpha} \right).
\]
There will be two values of $\alpha$ that are especially important.  The first is when $\alpha = 0$, which we use to bound the $L^2$ norm, while the second is for an arbitrary $\alpha$ that satisfies $1/2h < \alpha < 1 - 1/2h$, as this will be used to bound the last term in the next subsection.

\subsection*{Step 3: $\hat \bbeta - \tilde \bbeta$}
Recall that $\hat \bbeta = (\bT_{nm} + \lambda \bI)^{-1} \bh_{nm}$ and $ \tilde \bbeta =  \bbeta_\lambda + (\bT+\lambda \bI)^{-1}(\bh_{nm} - \bT_{nm}(\bbeta_\lambda)  - \lambda \bbeta_\lambda)$.  Note that this also implies that $\bh_{nm} = (\bT_{nm} + \lambda \bI) \hat \bbeta$. So write
\begin{align*}
\hat \bbeta - \tilde \bbeta
& = \hat \bbeta - \bbeta_\lambda -  (\bT+\lambda \bI)^{-1}(\bh_{nm} - \bT_{nm}(\bbeta_\lambda)  - \lambda \bbeta_\lambda) \\
& =   (\bT+\lambda \bI)^{-1}
\left(
 (\bT+\lambda \bI)( \hat \bbeta - \bbeta_\lambda)
- (\bh_{nm} - (\lambda \bI + \bT_{nm})\bbeta_\lambda) ) 
\right) \\
&=  (\bT+\lambda \bI)^{-1}
\left(
 (\bT+\lambda \bI)( \hat \bbeta - \bbeta_\lambda)
- (\bT_{nm} + \lambda \bI) (\hat \bbeta - \bbeta_\lambda )
\right).
\end{align*}
Computing the $\alpha$ norm we can apply Parseval's and the definition of $\bT_{nm}$ to obtain
\begin{align*}
& \| \hat \bbeta - \tilde \bbeta\|^2_\alpha \\
& = 
\sum_p \sum_k \frac{\tau_k^{-\alpha}}{(\eta_p\tau_k+\lambda )^{2}} \left[
(\tau_k+ \lambda) \langle \hat \bbeta - \bbeta_\lambda, \bu_p\otimes v_k \rangle
- \langle (\bT_{nm} + \lambda \bI) (\hat \bbeta - \bbeta_\lambda ), \bu_p \otimes v_k \rangle
\right]^2 \\
& = \sum_p\sum_k \frac{\tau_k^{2-\alpha}}{(\eta_p \tau_k+\lambda )^{2}} \left[
 \langle \hat \bbeta - \bbeta_\lambda, \bu_p \otimes v_k \rangle
- \frac{ 1}{n} \sum_{i=1}^n \frac{1}{m_i} \sum_{j=1}^{m_i}  \bu_p^\top (\hat \bbeta(u_{ij}) - \bbeta_\lambda(u_{ij} )) v_k(u_{ij}) 
\right]^2.
\end{align*}

Notice that we can write $\hat \bbeta(u) - \bbeta_\lambda(u) = \sum_{\ell=1}^\infty \bh_\ell v_\ell(u)$ where $h_{\ell p} = \langle \hat g_p - g_{\lambda,p}, v_\ell \rangle$.  We can then write
\[
\bu_p^\top (\hat \bbeta(u_{ij}) - \bbeta_\lambda(u_{ij} )) v_k(u_{ij}) = \sum_{\ell=1}^\infty \bu_p^\top \bh_\ell v_\ell(u_{ij}) v_k(u_{ij}).
\]
So the difference is given by
\begin{align*}
& \langle \hat \bbeta - \bbeta_\lambda, \bu_p \otimes v_k \rangle
- \frac{ 1}{n} \sum_{i=1}^n \frac{1}{m_i} \sum_{j=1}^{m_i}  \bu_p^\top (\hat \bbeta(u_{ij}) - \bbeta_\lambda(u_{ij} )) v_k(u_{ij})   \\
& =  \bu_p^\top \bh_k
- \frac{ 1}{n} \sum_{i=1}^n \frac{1}{m_i} \sum_{j=1}^{m_i} \sum_{\ell=1}^\infty  \bu_p^\top \bh_\ell v_\ell(u_{ij}) v_k(u_{ij}) \\
& = \sum_{\ell=1}^\infty \bu_p^\top \bh_\ell \left[\langle v_k, v_\ell\rangle - \frac{ 1}{n} \sum_{i=1}^n \frac{1}{m_i} \sum_{j=1}^{m_i} v_\ell(u_{ij}) v_k(u_{ij}) \right].
\end{align*}
Let $\delta \in [0,1]$ be another constant similar, but potentially different from $\alpha$.  We can then apply CS to bound the above by
\begin{align*}
|\langle \hat \bbeta - \bbeta_\lambda, \bu_p \otimes v_k \rangle| &\leq  \left(\sum_{\ell=1}^\infty \frac{(\bu_p^\top \bh_\ell) ^2}{\tau_\ell^\delta} \right)
 \sum_{\ell=1}^\infty \tau_\ell^\delta \left[\langle v_k, v_\ell\rangle - \frac{ 1}{n} \sum_{i=1}^n \frac{1}{m_i} \sum_{j=1}^{m_i} v_\ell(u_{ij}) v_k(u_{ij}) \right]^2 \\
 & =  \| \bu_p^\top(\hat \bbeta - \bbeta_\lambda)\|_\delta^2 \sum_{\ell=1}^\infty \tau_\ell^\delta \left[\langle v_k, v_\ell\rangle - \frac{ 1}{n} \sum_{i=1}^n \frac{1}{m_i} \sum_{j=1}^{m_i} v_\ell(u_{ij}) v_k(u_{ij}) \right]^2.
\end{align*}
To get the asymptotic order of the summation term above, by Markov's inequality, it is enough to bound its expected value (since it is positive) . Taking the expected value of the summation we get that
\begin{align*}
 & \sum_{\ell=1}^\infty \tau_\ell^\delta \E \left[\langle v_k, v_\ell\rangle - \frac{ 1}{n} \sum_{i=1}^n \frac{1}{m_i} \sum_{j=1}^{m_i} v_\ell(u_{ij}) v_k(u_{ij}) \right]^2 \\
 & =  \sum_{\ell=1}^\infty \frac{\tau_\ell^\delta}{nm} \Var( v_\ell(u_{11}) v_k(u_{11}) )\\
 & \leq \sum_{\ell=1}^\infty \frac{\tau_\ell^\delta}{nm} \int v_\ell(u)^2 v_k(u)^2 \ d \mu (u) 
 \leq \sum_{\ell=1}^\infty \frac{\tau_\ell^\delta}{nm} \sup_u v_k(u)^2 \int v_\ell(u)^2 \ du
 \leq \sum_{\ell=1}^\infty \frac{c_0 \tau_\ell^\delta}{nm \tau_k}. 
\end{align*}
Recall that $\tau_\ell \asymp \ell^{- 2h}$, so the above sum is finite as long as $\delta > 1/2h$.  Putting everything together and applying \eqref{e:eigbound} we have the bound
\[
\| \hat \bbeta - \tilde \bbeta\|^2_\alpha
\leq O_P(1) \| \hat \bbeta - \bbeta_\lambda\|_\delta^2 \frac{c_0}{nm} \sum_k \frac{\tau_k^{1-\alpha}}{(\tau_k + \lambda)^2}
\asymp O_P(1) \| \hat \bbeta - \bbeta_\lambda\|_\delta^2 (nm)^{-1} \lambda^{-\alpha - 1/2h},
\]
which holds for any $0 \leq \alpha < 1-1/2h$ and any $\delta > 1/2h$.

Assume that $\lambda$ is such that $ (nm)^{-1} \lambda^{-\alpha - 1/2h} \to 0$, then it follows that $\| \hat \bbeta - \tilde \bbeta\|^2_\alpha = o_P(\| \hat \bbeta - \bbeta_\lambda\|_\delta^2)$.  A triangle inequality gives
\[
\| \tilde \bbeta - \bbeta_\lambda\|_\delta \geq \| \hat \bbeta - \bbeta_\lambda\|_\delta - \| \hat \bbeta - \tilde \bbeta\|_\delta = (1 + o_P(1)) \| \hat \bbeta - \bbeta_\lambda\|_\delta.
\]
This implies that
\[
 \| \hat \bbeta - \bbeta_\lambda\|_\delta = O_P(\| \tilde \bbeta - \bbeta_\lambda\|_\delta).
\]
Finally, take $\alpha = 0$ and $\delta > 1/2h$ then we have that
\begin{align*}
& \|\hat \bbeta - \tilde \bbeta\|^2 = O_P(1) (nm)^{-1} \lambda^{-1/2h} \| \tilde \bbeta - \bbeta_\lambda\|_\delta^2 \\
& = O_P(1) (nm)^{-1} \lambda^{-1/2h} [(nm)^{-1}\lambda^{-\delta-1/2h}+n^{-1}\lambda^{-\delta}].
\end{align*}
If we assume that $\lambda$ is such that $ (nm)^{-1} \lambda^{-\delta - 1/2h} \to 0$ then the above simplifies to
\[
o_P(1) \lambda^\delta [(nm)^{-1}\lambda^{-\delta-1/2h}+n^{-1}\lambda^{-\delta}]
= o_P(1) [(nm)^{-1}\lambda^{-1/2h}+n^{-1}],
\]
as desired.

Note that in the last paragraph, we made a more explicit assumption about how quickly $\lambda$ tends to zero.  Note that the optimal rate is $\lambda = (nm)^{2h/(1+2h) }$.  For this value of $\lambda$ we have that $ (nm)^{-1} \lambda^{-\alpha - 1/2h} \to 0$ for any value of $\alpha < 1$ since $ 1 + 1/2h = (2h + 1)/2h$.

\end{document}